\documentclass[12pt]{amsart}

 \usepackage{amsfonts,graphics,amsmath,amsthm,amsfonts,amscd,amssymb,amsmath,latexsym,multicol,
 mathrsfs}
\usepackage{epsfig,url}
\usepackage{flafter}
\usepackage{fancyhdr}
\usepackage{hyperref}
\hypersetup{colorlinks=true, linkcolor=black}

\makeatletter

\def\jobis#1{FF\fi
  \def\predicate{#1}%
  \edef\predicate{\expandafter\strip@prefix\meaning\predicate}%
  \edef\job{\jobname}%
  \ifx\job\predicate
}

\makeatother

\if\jobis{proposal}%
\else
\fi

 \usepackage[matrix, arrow]{xy}

\DeclareMathOperator{\Supp}{Supp}

\DeclareMathOperator{\vol}{vol}

 \numberwithin{equation}{subsection}
 \numberwithin{footnote}{subsection}

 \newtheorem{cor}[subsection]{Corollary}
 \newtheorem{lem}[subsection]{Lemma}
 \newtheorem{prop}[subsection]{Proposition}
 \newtheorem{thm}[subsection]{Theorem}
 \newtheorem{conj}[subsection]{Conjecture}

{
    \newtheoremstyle{upright}%
        {8pt plus2pt minus4pt}%
        {8pt plus2pt minus4pt}%
        {\upshape}%
        {}%
        {\bfseries\scshape}%
        {}%
        {1em}%
        {}%
\theoremstyle{upright}

 \newtheorem{rem}[subsection]{Remark}

}

\newcommand{\xHom}{\mathscr{H}om}
 \newcommand{\x}{\mathscr}

 \newcommand{\PP}{\mathbb P}
 
 \newcommand{\Q}{\mathbb Q}
 \newcommand{\R}{\mathbb R}
 
 \newcommand{\bir}{\dashrightarrow}
 \newcommand{\rddown}[1]{\left\lfloor{#1}\right\rfloor} 

\title{Singularities on the base of a Fano type fibration}
\thanks{2010 MSC: 14E30, 14J17}
\author{Caucher Birkar}
\date{\today}
\begin{document}
\maketitle

\begin{abstract}
Let $f\colon X\to Z$ be a Mori fibre space. M$^{\rm c}$Kernan conjectured that the 
singularities of $Z$ are bounded in terms of the singularities of $X$. Shokurov generalised this 
to pairs: let $(X,B)$ be a klt pair and $f\colon X\to Z$ a contraction such that $K_X+B\sim_\R 0/Z$ and 
that the general fibres of $f$ are Fano type varieties; adjunction for fibre spaces produces a 
discriminant divisor $B_Z$ and a moduli divisor $M_Z$ on $Z$.   
it is then conjectured that the singularities of $(Z,B_Z+M_Z)$ are bounded in terms of the 
singularities of $(X,B)$. We prove Shokurov conjecture when $(F,\Supp B_F)$ belongs 
to a bounded family where $F$ is a general fibre of $f$ and $K_F+B_F=(K_X+B)|_F$. 
\end{abstract}



\section{Introduction}

We work over an algebraically closed field $k$ of characteristic zero. Let $X$ be a variety with klt singularities and 
$f\colon X\to Z$ a $K_X$-negative extremal contraction. When $f$ is a divisorial contraction, we can check that 
the singularities of $Z$ are as good as the singularities of $X$. More precisely, in terms of 
log discrepancies we have 
$$
a(E,X,0)\le a(E,Z,0)
$$ 
for every prime divisor $E$ on birational models of $X,Z$.  If $f$ is a flipping contraction, and if $f^+\colon X^+\to Z$ 
is the positive side of the flip, then singularities of $X^+$ are as good as the singularities of $X$.
This is important since 
if we want to prove a statement about $X$ we can often translate it into a similar statement about $Z$ or $X^+$ (eg, finite generation). 

If $f$ is not birational, one would still like to understand the singularities on 
$Z$ although this is much more complicated. Beside being an interesting problem on its own, it is also important for inductive arguments.
M$^{\rm c}$Kernan conjectured that the singularities of $Z$ are bounded in terms of the singularities of $X$, that is: 

\begin{conj}[$M_{d,\epsilon}$]\label{c-mc}
Let $d$ be a natural number and $\epsilon>0$ a real number.
Then, there is a real number $\delta>0$ depending on $d,\epsilon$ satisfying the following: 
let $f\colon X\to Z$ be a $K_X$-negative extremal contraction such that 

$\bullet$ $X$ is $\epsilon$-lc of dimension $d$ and $\Q$-factorial, and 

$\bullet$ $\dim X>\dim Z$.\\\\
Then $Z$ is $\delta$-lc.
\end{conj}

See \ref{s-pairs} for the definition of $\epsilon$-lc singularities.
When $d=1$ or $d=2$, the conjecture is trivial since $Z$ would be a smooth curve or just a point. 
Mori and Prokhorov [\ref{MP2}, Theorem 1.2.7] proved the conjecture for $d=3$ and $\epsilon=1$ but 
with $X$ having terminal singularities: 
in this case one can take $\delta=1$. Much more recently, Alexeev and Borisov 
[\ref{AB}] proved the conjecture for toric morphisms of toric varieties.

Shokurov generalised the conjecture to the setting of pairs. 
Let $f\colon X\to Z$ be a contraction of normal varieties, and
$(X,B)$ klt such that $K_X+B\sim_\R 0/Z$. By a construction of Kawamata [\ref{ka97}][\ref{ka98}] we may write
$$
K_X+B\sim_\R f^*(K_Z+B_Z+M_Z)
$$
where $B_Z$ is called the \emph{discriminant part} and 
$M_Z$ is called the \emph{moduli part}. The discriminant part is canonically determined as a Weil $\R$-divisor by the 
singularities of $(X,B)$ and the fibres over codimension one points of $Z$; the moduli part is then 
automatically determined as an $\R$-linear equivalence class, in particular, $M_Z$ 
may be represented by many different Weil $\R$-divisors. See \ref{s-adjunction} for more details. 

We are ready to state a refined version of Shokurov's conjecture.

\begin{conj}[$S_{d,\epsilon, \mathcal{P}}$]
Let $d$ be a natural number, $\epsilon>0$ a  real number, and $\mathcal{P}$ a set of couples.
Then, there is a real number $\delta>0$ depending on $d,\epsilon, \mathcal{P}$ satisfying the following: 
let $(X,B)$ be a pair and $f\colon X\to Z$ a contraction such that 

$\bullet$ $(X,B)$ is $\epsilon$-lc of dimension $d$, 

$\bullet$ $K_X+B\sim_\R 0/Z$,

$\bullet$ the general fibres $F$ of $f$ are of Fano type, 

$\bullet$  $(F,\Supp B_F)$ is isomorphic in codimension one with some $(F',D_{F'})\in\mathcal{P}$ where 
$K_F+B_F=(K_X+B)|_F$.\\\\
Then,  we can choose an $\R$-divisor $M_Z\ge 0$ representing the moduli part so that  $(Z,B_Z+M_Z)$ is $\delta$-lc.
\end{conj}

See \ref{s-pairs}, \ref{s-Fano-type}, and \ref{s-bnd-couples} for the definition of general fibres, Fano type varieties, 
couples and their boundedness. A couple is essentially a pair but with no condition on singularities 
except normality. Note that unlike in Conjecture \ref{c-mc}, $f$ is allowed to be a divisorial 
contraction, a flipping contraction, or a fibre type contraction. Also note that we are not assuming 
$X,Z$ to be projective although $f$ is projective.
Mori and Prokhorov  [\ref{MP}, Theorem 1.1] prove a result on weak del Pezzo fibrations 
in dimension $3$ which is closely related to the conjecture when 
$d=3$, $\epsilon=1$, $\dim Z=1$, and $-K_F$ is nef and big. 
 
Now we come to the main theorem of this paper.

\begin{thm}\label{t-main}
Shokurov Conjecture $S_{d,\epsilon, \mathcal{P}}$ holds if $\mathcal{P}$ is a 
bounded family of couples.
\end{thm}

To prove the theorem, we use  
a recent result of Hacon, M$^{\rm c}$Kernan, and Xu [\ref{HMX}, Theorem 1.3] on volumes of big log divisors.
In view of the theorem, it is natural to consider interesting cases of bounded $\mathcal{P}$    
(Corollaries \ref{c-2} and \ref{c-main}) 
and to try to reduce the conjecture to the theorem when $\mathcal{P}$ is not bounded 
(proof of Corollary \ref{c-dim2} and Remark \ref{rem-not-bnd}).  

Let $d$ be a natural number and $\epsilon, \lambda>0$ be real numbers. Consider the pairs $(F,B_F)$ satisfying:

$\bullet$ $(F,B_F)$ is $\epsilon$-lc and of dimension $\le d-1$, 

$\bullet$ $K_F+B_F\sim_\R 0$, 

$\bullet$ $F$ is of Fano-type,

$\bullet$ each non-zero coefficient of $B_F$ is $\ge \lambda$.

{\flushleft Let} $\mathcal{R}$ be the set of the couples $(F,\Supp B_F)$. 
It is expected that  $\mathcal{R}$ is a bounded family.
This boundedness is known when $d\le 3$ (see Theorem \ref{t-bnd-surfaces}). 
Actually, if one tries to prove the  boundedness in any dimension, then Theorem \ref{t-main} 
appears naturally (see Remark \ref{rem-bnd-R} for a discussion on this).

\begin{cor}\label{c-2}
Conjecture $S_{d,\epsilon, \mathcal{R}}$ holds for the above data $d,\epsilon,\mathcal{R}$ when $d\le 3$.
\end{cor}

Under some extra assumptions, the boundedness of $\mathcal{R}$ is known in any dimension.
More precisely: 
let $d$ be a natural number, $\epsilon>0$ a real number, and $\Lambda\subset [0,1]$ a finite set 
of real numbers. Consider the pairs $(F,B_F)$ satisfying:

$\bullet$ $(F,B_F)$ is $\epsilon$-lc and of dimension $\le d-1$, 

$\bullet$ $K_F+B_F\sim_\R 0$,

$\bullet$ $F$ is projective and $-K_F$ is ample, i.e. $F$ is a Fano variety,

$\bullet$ the coefficients of $B_F$ belong to $\Lambda$.

{\flushleft Let} $\mathcal{Q}$ be the set of the couples $(F,\Supp B_F)$. 
By  [\ref{HMX}, Corollary 1.7],  $\mathcal{Q}$ is a bounded family.

\begin{cor}\label{c-main}
Conjecture $S_{d,\epsilon, \mathcal{Q}}$ holds for the above data $d,\epsilon,\mathcal{Q}$.
\end{cor}

For surfaces we can verify $S_{d,\epsilon, \mathcal{P}}$ without boundedness assumptions:

\begin{cor}\label{c-dim2}
Conjecture $S_{2,\epsilon, \mathcal{P}}$ holds. More generally:
Conjecture $S_{d,\epsilon, \mathcal{P}}$ holds for those $(X,B)$ and $f\colon X\to Z$ 
with $d\le \dim Z+ 1$. \\
\end{cor}

We say a few words about the proof of Theorem \ref{t-main}. 
The difficult part of the theorem is to deal with the discriminant part $B_Z$ since 
by applying a result of Ambro [\ref{am05}] we can control the moduli part $M_Z$ 
(actually we have to understand the discriminant b-divisor $\mathcal{B}_Z$ rather than just 
$B_Z$).  
By taking hyperplane sections of $Z$ one can reduce the problem to the case $\dim Z=1$. 
Here one is mainly concerned about bounding the multiplicities of each fibre of some fibration 
birational to $f$.
Mori and Prokhorov  [\ref{MP}, Theorem 1.1] do this by using the orbifold Riemann-Roch theorem 
for varieties of dimension $3$ with terminal singularities. Unfortunately, this approach 
does not work in higher dimensions. Instead of Riemann-Roch, we use  
the boundedness of volumes of big log divisors [\ref{HMX}, Theorem 1.3].

Shokurov has another approach to the theorem: as far as I understand he is trying to construct a compactified 
coarse moduli 
space for the log general fibres and then recover information about $f$ from the moduli space. 
Our approach is more direct and it does not rely on the existence of such a moduli space. 
However, the techniques developed in this paper might in fact be useful to construct such moduli spaces.\\ 

\textbf{Acknowledgements:} This work was supported by a grant from Leverhulme. 
Thanks to Yuri Prokhorov for useful discussions on his results with Mori.


\vspace{0.5cm}
\section{Preliminaries}

\subsection{Pairs.}\label{s-pairs}
We work over an algebraically closed field $k$ of characteristic zero: all the varieties 
and schemes are over $k$ unless stated otherwise. 
A \emph{sub-pair} $(X,B)$ consists of a normal quasi-projective variety $X$ and  
a sub-boundary $B$, that is, an $\R$-divisor on $X$ with
coefficients in $(-\infty,1]$ such that $K_X+B$ is $\mathbb{R}$-Cartier. 
For a prime divisor $D$ on some birational model of $X$ with a
nonempty centre on $X$, $a(D,X,B)$
denotes the log discrepancy. We say that $(X,B)$ is \emph{$\epsilon$-lc} if $a(D,X,B)\ge \epsilon$ for every 
prime divisor $D$ on birational models of $X$ (if $B=0$, we also say that $X$ is {$\epsilon$-lc}). 
This is equivalent to the following: let $g\colon Y\to X$ be any 
projective birational morphism from a normal variety $Y$ and write $K_Y+B_Y:=g^*(K_X+B)$; then every coefficient of 
$B_Y$ is $\le 1-\epsilon$.  

We call a sub-pair $(X,B)$ \emph{lc} if it is $0$-lc. We call it \emph{klt} if it is $\epsilon$-lc 
for some $\epsilon>0$.
A sub-pair $(X,B)$ is called a \emph{pair} if the coefficients of $B$ are non-negative: in this 
case we call $B$ a boundary. We refer to [\ref{KM}] for standard definitions and results 
on singularities of pairs and the log minimal model program.

Let $(X,B)$ be a lc sub-pair and $M\ge 0$ an $\R$-Cartier divisor. The \emph{lc threshold} of 
$M$ with respect to $(X,B)$ is the largest real number $t$ so that $(X,B+tM)$ is lc.

A \emph{contraction} $f\colon X\to Z$ is a projective morphism of quasi-projective varieties 
 with $f_*\x{O}_X=\x{O}_Z$. A general fibre of $f$ is a fibre over a closed point belonging 
to some fixed open set $U\subset Z$. In practice, $U$ is understood from the context 
and we might shrink it without mention. 
If $(X,B)$ is a pair, a \emph{log general fibre} of $(X,B)$ and 
$f$ is as $(F,B_F)$ where $F$ is a general fibre of $f$ and $K_F+B_F=(K_X+B)|_F$.

\subsection{Minimal models and Mori fibre spaces.}\label{s-m-models}
Let $(X,B)$ be a lc pair, $(Y,B_Y)$ a $\Q$-factorial dlt pair, $X\to Z$ and $Y\to Z$ 
contractions,  $\phi\colon X\bir Y/Z$ a birational 
map such that $\phi^{-1}$ does not contract divisors, and $B_Y=\phi_*B$.   Moreover, assume that 
$$
a(D,X,B)\le a(D,Y,B_Y)
$$ 
for any prime divisor $D$ on birational models of $X$ and assume that the strict inequality holds for 
any prime divisor $D$ on $X$ which is exceptional/$Y$.
We say that $(Y,B_Y)$ is a \emph{log minimal model} of $(X,B)$ over $Z$ if $K_Y+B_Y$ is nef$/Z$. 
On the other hand, we say that $(Y,B_Y)$ is a \emph{Mori fibre space} of $(X,B)$ over $Z$ if 
there is a $K_Y+B_Y$-negative extremal contraction $Y\to Y'/Z$ such that $\dim Y'<\dim Y$. 
These definitions follow the traditional definitions of log minimal models and Mori fibre spaces. 
The corresponding definitions in [\ref{B-lc-flips}] are more general but we do not need such 
generality in this paper.

Assume that $(X,B)$ is a $\Q$-factorial klt pair and $f\colon X\to Z$ a contraction with $K_X+B$ or $B$ big$/Z$.  
Also assume that $C\ge 0$, $(X,B+C)$ is klt, and $K_X+B+C$ is nef$/Z$. 
Then by [\ref{BCHM}], any LMMP$/Z$ on $K_X+B$ with scaling of $C$ 
ends up with a log minimal model or a Mori fibre space over $Z$.

\subsection{Fano type varieties.}\label{s-Fano-type}
A projective variety $X$ is said to be of \emph{Fano type} if there is a boundary 
$C$ such that $(X,C)$ is a klt pair and $-(K_X+C)$ is ample. Note that if there is 
another boundary $B$ with $K_X+B\sim_\R 0$, then $B$ is big because $-K_X$ is big.

\begin{lem}\label{l-Fano-type}
Assume that $\phi\colon X\bir X'$ is an isomorphism in codimension one between normal 
projective varieties. If $X$ is of Fano type, then $X'$ is also of Fano type.
\end{lem}
\begin{proof}
There is a boundary $C$ such that $(X,C)$ is a klt pair and $-(K_X+C)$ is ample. 
There is an $\R$-Cartier divisor $D$ such that $(X,C+D)$ is klt and $K_X+C+D\sim_\R 0$.
Obviously, $D$ is ample. Now, $(X',C'+D')$ is klt, $K_{X'}+C'+D'\sim_\R 0$, and 
$D'$ is big where $C'$ denotes the birational transform of $C$ (similar notation for the other divisors). 
We can write $D'\sim_\R A'+G'$ where $A'$ is ample and $G'\ge 0$. 
Thus, 
$$
(X',C'+(1-t)D'+tA'+tG')
$$ 
is klt and 
$$
K_{X'}+C'+(1-t)D'+tA'+tG'\sim_\R 0
$$
if $t>0$ is sufficiently small. Put $\Delta'=C'+(1-t)D'+tG'$. Then, 
$(X',\Delta')$ is klt and $-(K_{X'}+\Delta')$ is ample hence $X'$ is also of Fano type.
\end{proof}

\subsection{B-divisors.}
Let $X$ be a normal variety. An \emph{$\R$-b-divisor} $\mathcal{D}$ on $X$ is a collection of 
$\R$-divisors $D_Y$ for each birational contraction $Y\to X$ so that if we have a birational 
contraction $\pi\colon Y'\to Y/X$, then $\pi_*D_{Y'}=D_Y$. In birational geometry, one often 
has to consider all the resolutions of a variety $X$ in order to understand properties of 
a divisor on $X$. This naturally leads to b-divisors which were defined by Shokurov [\ref{log-models}][\ref{pl-flips}].

\subsection{Adjunction for fibre spaces}\label{s-adjunction} 
Let $f\colon X\to Z$ be a contraction of normal varieties, and
$(X,B)$ a klt sub-pair such that $K_X+B\sim_\R 0/Z$, that is, we have an equivalence $K_X+B\sim_\R f^*N$ for
some $\R$-Cartier $\R$-divisor $N$. 
By a construction of Kawamata [\ref{ka97}][\ref{ka98}] we have
a decomposition
$$
N\sim_\R K_Z+B_Z+M_Z
$$
where
$B_Z$ is defined using the singularities of $(X,B)$ and of the fibres of $f$. 
The part $B_Z$ is called the \emph{discriminant part} and the
part $M_Z$ is called the \emph{moduli part}. 
 More precisely, $B_Z$ is defined as follows: 
for each prime divisor $D$ on $Z$, let $t$ be the lc threshold of $f^*D$ over the 
generic point of $D$, with respect 
to the pair $(X,B)$; then let $(1-t)$ be the coefficient of $D$ in $B_Z$. The moduli part $M_Z$ is 
then determined as an $\R$-linear equivalence class.

Consider a commutative diagram
$$
\xymatrix{
X' \ar[r]^{\tau} \ar[d]^{f'} &  X\ar[d]^{f}\\
Z' \ar[r]^{\sigma} & Z }
$$
in which $X',Z'$ are normal, $\sigma,\tau$ are birational contractions, and $f'$ is a contraction. 
Let $K_{X'}+B':=\tau^*(K_X+B)$. 
Using the relation $K_{X'}+B'\sim_\R 0/Z'$, we can similarly define a decomposition 
$$
\sigma^*N\sim_\R K_{Z'}+B_{Z'}+M_{Z'}
$$
which satisfies  $B_Z=\sigma_*B_{Z'}$ and $M_Z\sim_\R \sigma_*M_{Z'}$. Putting all the $B_{Z'}$ together for 
all the possible $Z'$ determines 
an $\R$-b-divisor $\mathcal{B}_Z$. We could also choose the $M_{Z'}$ consistently so that 
putting all the $M_{Z'}$ together we get an $\R$-b-divisor $\mathcal{M}_Z$.

Now assume that $X,Z$ are projective and that $B$ is effective with rational coefficients. 
Kawamata [\ref{ka97}][\ref{ka98}] showed that 
if $Z'$ is a sufficiently high resolution then $M_{Z'}$ is a nef $\Q$-divisor. 
Following ideas of Kawamata [\ref{ka85}], Ambro [\ref{am05}] proved that, perhaps after replacing $Z'$ 
with a higher resolution, $M_{Z'}$ satisfies a pullback property: 
for any other resolution $\pi\colon Z''\to Z'$ we have $M_{Z''}\sim_\Q \pi^*M_{Z'}$. Moreover, 
he showed that $M_{Z'}$ is the pullback of a nef and big divisor under some contraction $Z'\to T$.
We call such a $Z'$ an \emph{Ambro model}.

\subsection{Volume of divisors} 
Let $X$ be a normal projective variety of dimension $d$ and $D$ an integral divisor on $X$. The \emph{volume} of $D$ 
denoted by $\vol(D)$ is defined as 
$$
\vol(D)=\limsup_{m\to +\infty} \frac{h^0(\x{O}_X(mD))}{m^d/d!}
$$
In some places we also use the notation $\vol(\x{O}_X(D))$ instead of $\vol(D)$.

If $D$ is a $\Q$-divisor, we can define $\vol(D)=\frac{1}{m^d}\vol(mD)$ for some $m>0$ with $mD$ being integral. 
If $D$ is not big, then it is obvious that $\vol(D)=0$.
If $D$ is nef, then $\vol(D)=D^d$ which follows from the Riemann-Roch theorem.  
Ein-Lazarsfeld-Musta\c{t}\u{a}-Nakamaye-Popa [\ref{ELMNP}] treat this topic in detail.

Assume that $\Omega$ is a set of rational numbers satisfying the descending chain condition 
(DCC). Hacon-M$^{\rm c}$Kernan-Xu [\ref{HMX}, Theorem 1.3] (see also [\ref{HMX2}])
proved that there is a number $\theta>0$ depending only on $d$ and $\Omega$ such that:
if $(X,B)$ is a projective lc pair of dimension at most $d$, if the coefficients of $B$ belong 
to $\Omega$, and if $K_X+B$ is big, then $\vol(K_X+B)\ge \theta$. We will apply this in the proof of 
Proposition \ref{p-2}.

\subsection{Couples and bounded families}\label{s-bnd-couples}
A \emph{couple} $(F,D_F)$ consists of a normal projective variety $F$ and a  divisor 
$D_F$ on $F$ whose coefficients are all equal to $1$, i.e. $D_F$ is a reduced divisor. 
The reason we call $(F,D_F)$ a couple rather than a pair is that we are concerned with 
$D_F$ rather than $K_F+D_F$ and we do not want to assume that  $K_F+D_F$ is $\Q$-Cartier or 
that it has nice singularities.

Two couples  $(F,D_F)$ and $(F',D_{F'})$ 
are \emph{isomorphic in codimension one} if there is a birational isomorphism 
$\phi\colon F\bir F'$ which is an isomorphism in codimension one such that $\phi_*D_F=D_{F'}$.
The two couples are \emph{isomorphic} if $\phi$ is an isomorphism. 

We say that a set of couples $\mathcal{P}$ is \emph{a bounded family} if there is a projective morphism 
$f\colon S\to T$ of Noetherian schemes over $k$ and a closed subscheme $D$ of $S$ 
such that for each $(F,D_F)\in \mathcal{P}$ there is a closed point $t\in T$ and an isomorphism between 
$(F,D_F)$ and $(S_t,\tilde{D}_t)$ 
where $S_t,D_t$ are the fibres over $t$ of the morphisms $S\to T$ and $D\to T$ respectively and 
$\tilde{D}_t\le D_t$: in particular, 
for such $t$ the scheme $S_t$ is a normal projective variety and the subscheme $D_t$ 
is a reduced divisor on $S_t$.

\begin{lem}\label{l-bnd-volume}
Let $\mathcal{P}$ be a bounded family of couples. Then, there is a natural number $v$ depending 
only on $\mathcal{P}$ such that for any $(F,D_F)\in \mathcal{P}$ the volume 
$\vol(D_F)\le v$.
\end{lem}
\begin{proof}
As a general principle, the numerical invariants of a bounded family are bounded. 
We will give a detailed proof. Let $f\colon S\to T$ be as in 
the definition of bounded family of couples. Let $T'$ be the reduced scheme associated to 
$T$ and let $S'=S\times_TT'$ and $D'=D\times_TT'$. Let $t$ be a closed point of $T$ and 
$S_t$ and $D_t$ the fibres over $t$ of the morphisms $S\to T$ and $D\to T$. 
These fibres are isomorphic to the corresponding fibres of the morphisms $S'\to T'$ 
and $D'\to T'$. Thus, we could assume from the beginning that $T$ is a reduced scheme. 
Moreover, we may assume that $T$ is affine by the Noetherian property.

Fix a closed embedding $S\to \PP^n_T/T$ and let $\x{O}_S(1)$ be the inverse image of $\x{O}_{\PP^n_T}(1)$. 
Since $S\to T$ is projective and $T$ affine, there is a coherent locally free sheaf $\x{E}$ which is a direct sum of invertible sheaves 
of the form $\x{O}_S(-m)$ with $m\gg 0$ admitting a surjective morphism $\x{E}\to \x{I}_D$ 
where $\x{I}_D$ is the ideal sheaf of $D$ in $S$. This gives an exact sequence $\x{E}\to \x{O}_S\to \x{O}_D\to 0$. 

Fix a closed point $t\in T$ such that the fibre $S_t$ is a normal variety and that $D_t$ is a reduced 
divisor on $S_t$. Restricting the above 
exact sequence to $S_t$ we get an exact sequence $\x{E}_t\to \x{O}_{S_t}\to \x{O}_{D_t}\to 0$. 
In particular, we have a surjection $\x{E}_t\to \x{I}_{D_t}$ where $\x{I}_{D_t}$ is the ideal sheaf 
of $D_t$ in $S_t$. 
We then get an injection 
$\x{I}_{D_t}^\vee \to \x{E}_t^\vee$ where for a coherent sheaf $\x{N}$ on $S_t$ we define
$\x{N}^\vee:=\xHom_{\x{O}_{S_t}}(\x{N},\x{O}_{S_t})$. 
Let $U$ be the smooth locus of $S_t$. Then, 
$\x{I}_{D_t}|_U=\x{O}_U(-D_t)$ hence $\x{I}_{D_t}^\vee|_U=\x{O}_U(D_t)$.
Since $\x{E}_t^\vee$ is reflexive, the injection $\x{I}_{D_t}^\vee|_U \to \x{E}_t^\vee|_U$
induces an injection $\x{O}_{S_t}(D_t)\to \x{E}_t^\vee$.

Let 
$$
0\to \x{G}'\to \x{G} \to \x{G}''\to 0
$$ 
be an exact sequence of coherent locally free sheaves 
on $S_t$ such that we are given an injection $\x{O}_{S_t}(D_t)\to \x{G}$. Then, by restricting to 
$U$ one can see that either the induced morphism $\x{O}_{S_t}(D_t)\to \x{G}''$ is injective or $\x{O}_{S_t}(D_t)$   
is mapped into the kernel of $\x{G} \to \x{G}''$ in which case we get an injection $\x{O}_{S_t}(D_t)\to \x{G}'$.

 By construction, $\x{E}_t^\vee$ is a coherent locally free sheaf which is a direct sum of 
sheaves of the form $\x{O}_{S_t}(m)$ with $m\gg 0$. Applying the last paragraph we get an injection 
$\x{O}_{S_t}(D_t)\to \x{O}_{S_t}(m_t)$ for some $m_t\gg 0$ where there are only finitely many possibilities for 
$m_t$ for all $t$ as above. 

Now, since $T$ is reduced, by the generic flatness and stratification theorem, there are only finitely 
many possibilities for the Hilbert polynomial  $\Phi_t$ of the fibre $S_t$. By definition, 
$\Phi_t(m)=\mathcal{X}(\x{O}_{S_t}(m))$. On the other hand, since $m_t$ is sufficiently large, 
$h^i(\x{O}_{S_t}(mm_t))=0$ for any $i,m>0$. Thus, $\Phi_t(mm_t)=h^0(\x{O}_{S_t}(mm_t))$ for $m>0$ and  
this in turn implies that the volume
$$
\vol(\x{O}_{S_t}(m_t))=\limsup_{m\to +\infty} \frac{h^0(\x{O}_{S_t}(mm_t))}{m^{\dim S_t}/\dim S_t!}
=\limsup_{m\to +\infty} \frac{\Phi_t(mm_t)}{m^{\dim S_t}/\dim S_t!}
$$ 
depends only on $\Phi_t$ and $m_t$. Since there are only finitely many possibilities for 
$\Phi_t$ and $m_t$, there is a natural number $v$ such that 
$\vol(\x{O}_{S_t}(m_t))\le v$ for every $t$ as above. 
Thus, $\vol(D_t)\le v$ 
because of the injection $\x{O}_{S_t}(D_t)\to \x{O}_{S_t}(m_t)$. By definition, each 
$(F,D_F)\in \mathcal{P}$ is isomorphic to some $(S_t,\tilde{D}_t)$  with
$\tilde{D}_t\le D_t$. Therefore, 
$$
\vol(D_F)=\vol(\tilde{D}_t)\le \vol(D_t)\le v
$$
and we are done.
\end{proof}

\subsection{Intersection numbers} 
In certain cases we will apply the following lemma to compare intersection numbers on a variety which may not be 
proper.

\begin{lem}\label{l-intersection}
Let $h\colon Y\to Z$ be a contraction from a normal variety $Y$ of dimension $d$ to a smooth curve $Z$. 
Assume that $L$ is a $\Q$-Cartier divisor on $Y$ which is nef$/Z$.
Pick two distinct closed points $P,Q\in Z$ and write $h^*P=\sum m_iT_i$ and $h^*Q=\sum n_jS_j$ where $T_i,S_j$ 
are the irreducible components. 
Then, 
$$
\sum m_i(L|_{T_i})^{d-1}=\sum n_j(L|_{S_j})^{d-1}
$$
\end{lem}
\begin{proof}
Take a resolution $\phi\colon W\to Y$. We can write $\phi^*h^*P=\sum m_iT_i'+\sum l_kE_k$ where 
$E_k$ are prime exceptional divisors of $\phi$ and $T_i'$ is the birational transform of $T_i$. 
Since $\phi^*L|_{E_k}$ is not big, $(\phi^*L|_{E_k})^{d-1}=0$. On the other hand, since the induced 
morphism $T_i'\to T_i$ is birational  $(\phi^*L|_{T_i'})^{d-1}=(L|_{T_i})^{d-1}$. Thus, 
$$
\sum m_i(L|_{T_i})^{d-1}=\sum m_i(\phi^*L|_{T_i'})^{d-1}+\sum l_k(\phi^*L|_{E_k})^{d-1}
$$
A similar equality holds for $\phi^*h^*Q$. Therefore, by replacing $X$ with $W$ and $L$ with $\phi^*$ 
we may assume that $X$ is already smooth. 

By taking compactifications and then a resolution we can assume that there is a contraction 
$\overline{f}\colon \overline{Y}\to \overline{Z}$ of smooth projective varieties 
such that $Y\subset \overline{Y}$,  $Z\subset \overline{Z}$, and that $\overline{f}|_{Y}=f$. 
By replacing $\overline{f}$ with $f$ and $L$ with its closure in $\overline{Y}$, 
we may assume that $Y,Z$ are projective. Now, since $P-Q\equiv 0$, $h^*P\equiv h^*Q$ hence intersection with the $1$-cycle $L^{d-1}$ would be the same which implies 
that 
$$
\sum m_i(L|_{T_i})^{d-1}= L^{d-1}\cdot (\sum m_iT_i)
=  L^{d-1}\cdot h^*P
= L^{d-1}\cdot h^*Q
$$
$$
= L^{d-1}\cdot (\sum n_jS_j)
=  \sum n_j(L|_{S_j})^{d-1}
$$

\end{proof}

\vspace{0.5cm}
\section{Bounding the discriminant b-divisor}

In this section we will bound the coefficients of the discriminant b-divisor $\mathcal{B}_Z$ 
that is associated to the data in Theorem \ref{t-main}.  
First we deal with the discriminant divisor $B_Z$ and later we take care of the b-divisor.

\begin{prop}\label{p-2}
Assume that $d,\epsilon, \mathcal{P}$ are as in Conjecture $S_{d,\epsilon, \mathcal{P}}$ and that 
$\mathcal{P}$ is a bounded family of couples. Then, there is a 
real number $\delta>0$ depending only on $d,\epsilon, \mathcal{P}$ such that if $(X,B)$ and $f\colon X\to Z$ 
are as in the conjecture, then every coefficient of ${B}_Z$ is $\le 1-\delta$. 
\end{prop}

By taking hyperplane sections on $Z$ we reduce the problem to the case $\dim Z=1$. 

\begin{lem}\label{l-dim-1}
Assume that Proposition \ref{p-2} holds for the data $d-1,\epsilon, \mathcal{P}$. 
Then the proposition holds for the data $d,\epsilon, \mathcal{P}$ for those $(X,B)$ and $f\colon X\to Z$ with 
$\dim Z>1$. 
\end{lem}
\begin{proof}
Assume that $\delta>0$ is as in the proposition for the data $d-1,\epsilon, \mathcal{P}$.
Let $(X,B)$ and $f\colon X\to Z$ be as in the proposition for the data $d,\epsilon, \mathcal{P}$, 
and assume that $\dim Z>1$. 
Let $D$ be a component of $B_Z$, 
and $t$ the lc threshold of $f^*D$ over the generic point of $D$ with respect to the pair $(X,B)$. 
We will show that $t\ge \delta$ which means that the coefficient of $D$ in $B_Z$ is at most $1-\delta$. 
By removing some codimension $2$ closed subset of $Z$ we may assume that $Z$ is smooth.  
Moreover, we can assume that $(X,B+tf^*D)$ is lc whose lc centres all map onto $D$. 
By definition, $(X,B+tf^*D)$ has at least one lc centre.
 
Pick a general hyperplane section $H\subset Z$ (which would intersect $D$) and let $G=f^*H$.  
This ensures that  $(X,B+G+tf^*D)$ is lc and that $(X,B+G)$ is $\epsilon$-lc in codimension $\ge 2$, that is, 
$a(P,X,B+G)\ge \epsilon$ for any prime divisor $P$ exceptional$/X$. 
Letting 
$$
K_G+B_G=(K_X+B+G)|_G
$$ 
we get a klt pair $(G,B_G)$ and a contraction 
$g\colon G\to H$ which satisfy properties of the proposition for the data $d-1,\epsilon, \mathcal{P}$: 
indeed, $(G,B_G)$ is $\epsilon$-lc of dimension $d-1$, $K_G+B_G\sim_\R 0/H$, and the log general fibres of $g$ are among 
the log general fibres of $f$. 

By further shrinking $Z$ around $D$ 
we can assume that $D_H:=D\cap H$ is irreducible.
By construction,  
$(G,B_G+ tg^*D_H)$ is lc. Moreover, if $V$ is a lc centre of $(X,B+tf^*D)$ then $V\cap G\neq \emptyset$ 
because $V$ is mapped onto $D$ and $G$ contains every fibre over points of $D_H\subset D$. 
Since $V$ and $G$ are both lc centres of $(X,B+G+tf^*D)$, by Ambro [\ref{Ambro2}, Theorem 1.1]
the intersection $V\cap G$ is a union of lc centres of 
 $(X,B+G+tf^*D)$. Each lc centre of $(X,B+G+tf^*D)$ which sits inside $G$ is also a lc 
 centre of $(G,B_G+ tg^*D_H)$. Thus, $V\cap G$ is a union of lc centres of $(G,B_G+ tg^*D_H)$.
As $V\cap G$ maps onto $D_H$, there is a lc centre $W$ of $(G,B_G+ tg^*D_H)$ which maps onto $D_H$. 
In particular, this means that $t$ is the lc threshold of $g^*D_H$ with respect to the 
pair $(G,B_G)$ over the generic point of $D_H$.
By assumptions,  $t\ge \delta$. 
Therefore, the coefficient of $D$ in $B_Z$ is at most $1-\delta$ and we are done.      
\end{proof}

\begin{proof}(of Proposition \ref{p-2}) 
\emph{Step 1.}
Let $d,\epsilon,\mathcal{P}$ be as in the proposition where $\mathcal{P}$ is a bounded family by assumption. 
We will apply induction so we can assume that the proposition holds for the data 
$d',\epsilon,  \mathcal{P}$ if $d'<d$. Let $(X,B)$ and $f\colon X\to Z$ be as in the 
proposition for the data $d,\epsilon,\mathcal{P}$. 
By Lemma \ref{l-dim-1}, we may assume that $\dim Z=1$. By taking a $\Q$-factorialisation and applying Lemma 
\ref{l-Fano-type} we 
may assume that $X$ is $\Q$-factorial.

Fix a closed point $D\in Z$. We will find a real number $\delta>0$ depending only on 
$d,\epsilon, \mathcal{P}$ such that the coefficient of $D$ in $B_Z$ is $\le 1-\delta$.
We can shrink $Z$ around $D$ if necessary. In particular, we can assume that 
each component of $B$ is either horizontal$/Z$ or mapped to $D$. 

\emph{Step 2.} 
By decreasing $\epsilon$ if necessary we may assume that $\epsilon$ is rational.
Put $b=1-\frac{\epsilon}{2}$. Since $(X,B)$ is $\epsilon$-lc, 
and since for any prime divisor $P$ on $X$ we have $\epsilon\le a(P,X,B)\le 1$, $b> 0$.
Let $\phi\colon W\to X$ be a log resolution of $(X,B+f^*D)$. Let $\{M_i\}$ be the 
set of components of $\phi^*f^*D$, let $\{M_j'\}$ be the set of prime exceptional divisors of $\phi$ 
which do not belong to $\{M_i\}$, and let $\{M_k''\}$ be the set of components of 
the birational transform of $B$ which do not belong to $\{M_i\}$. 
Define 
$$
\Delta_W=\sum M_i+\sum bM_j'+\sum b_kM_k''
$$ 
where $b_k$ is the coefficient of $M_k''$ in the birational transform of $B$. 
Define 
$$
\Gamma_W=\sum M_i+\sum bM_j'+\sum bM_k''
$$

By construction, $(W,\Delta_W)$ and $(W,\Gamma_W)$ are both dlt, $\Supp \Gamma_W=\Supp \Delta_W$,  
$\Gamma_W-\Delta_W\ge 0$,
$$
\rddown{\Gamma_W}=\rddown{\Delta_W}=\Supp \phi^*f^*D,
$$ 
 each component of 
$\Gamma_W-\Delta_W$ is a component of the birational transform of $B$ and 
horizontal$/Z$, and  
$\Gamma_W-\Delta_W$ and $\rddown{\Gamma_W}=\rddown{\Delta_W}$ have no common components.
Also, note that the coefficients of $\Gamma_W$ belong to the set $\{b,1\}$.

Write $K_W+B_W=\phi^*(K_X+B)$. Since $(X,B)$ is $\epsilon$-lc, each coefficient of $B_W$ 
is at most $1-\epsilon$. If $Q$ is a component of $B_W$ with positive coefficient, then 
$Q$ is exceptional$/X$ or a component of the birational transform of $B$. In either case 
$Q$ is a component of $\Gamma_W$. 
So from $1-\epsilon<b$ we get $\Gamma_W-B_W\ge 0$ and that $\Supp \phi^*B\subseteq \Supp (\Gamma_W-B_W)$. 
Thus,  
 $B_W+r\phi^*B\le \Gamma_W$ if 
$r>0$ is sufficiently small. This implies that 
$$
\phi^*(K_X+B+rB)\le K_W+\Gamma_W
$$ 
hence $K_W+\Gamma_W$ is big$/Z$ because $K_X+B+rB$ is big$/Z$ 
which in turn follows from the assumptions that $K_X+B\sim_\R 0/Z$ and that 
the general fibres of $f$ are of Fano type (see \ref{s-Fano-type}). 

\emph{Step 3.}  
Let $\Gamma=\phi_*\Gamma_W$. We can write 
$$
K_W+\Gamma_W=\phi^*(K_X+\Gamma)+E_W
$$ 
where  $E_W$ is exceptional$/X$. Run the LMMP$/X$ on $K_W+\Gamma_W$ with scaling of some 
ample divisor. By [\ref{B-lc-flips}, Theorem 3.5], after finitely many steps we get a 
model $V$ on which $E_V\le 0$. Let $g$ be the morphism $V\to Z$. Let $(G,\Gamma_G)$ be a 
log general fibre of $(V,\Gamma_V)$ and $g$, and let $(F,\Gamma_F)$ be the corresponding log 
general fibre of $(X,\Gamma)$ and $f$. Then, 
$$
K_G+\Gamma_G=\psi^*(K_F+\Gamma_F)+E_G
$$
where $\psi$ is the morphism $G\to F$ and $E_G:=E_V|_G\le 0$.  

By construction, $\Supp \Gamma=\Supp B$ over the generic point of $Z$ hence $\Supp \Gamma_F=\Supp B_F$. 
By assumptions, the couple $(F,\Supp \Gamma_{F})=(F,\Supp B_{F})$ is isomorphic in codimension one 
with some couple $(F',D_{F'})\in \mathcal{P}$. By Lemma \ref{l-bnd-volume}, the volume of $D_{F'}$ 
is bounded by a number $v$ depending only on $\mathcal{P}$. Thus, the volume of $L_F:=\Supp \Gamma_{F}$ is also bounded by $v$ because 
$\vol(L_F)=\vol(D_{F'})$.
On the other hand, since  $K_X+B\sim_\R 0/Z$, there is a rational boundary $B'$ with the 
same support as $B$ such that 
 $K_X+B'\sim_\Q 0/Z$ (see the approximation arguments in the proof of Theorem \ref{t-main} below) 
 from which we get 
$$
K_{F}+\Gamma_F\le K_{F}+L_F =K_F+B_F'+L_F-B_F'\sim_\Q L_F-B_F'\le L_F
$$
hence the volume of $K_{F}+\Gamma_F$ is bounded by $v$. Since $E_G\le 0$, the volume of $K_G+\Gamma_G$ 
is also bounded by $v$.

\emph{Step 4.}
By construction, $\Supp \rddown{\Gamma_V}=\Supp g^*D$. Since $g^*D\sim 0/Z$, we can write 
$$
K_V+\Gamma_V\sim_\Q K_V+C_V/Z
$$ 
where $C_V\le \Gamma_V$, $(V,C_V)$ is klt and $K_V+C_V$ is big$/Z$. So, if we run the LMMP$/Z$ 
on $K_V+\Gamma_V$ with scaling of some ample divisor, then it terminates  
with a model $Y$ on which $K_Y+\Gamma_Y$ is semi-ample$/Z$. Let $h$ denote 
the morphism $Y\to Z$.

Let $(H,\Gamma_H)$ be a log general fibre of $(Y,\Gamma_Y)$ and $h$, and let $(G,\Gamma_G)$ be
 the corresponding log general fibre of $(V,\Gamma_V)$ and $g$. 
 By Step 3, $\vol(K_G+\Gamma_G)$ is bounded. Since 
$(Y,\Gamma_Y)$ is a log minimal model of $(V,\Gamma_V)$ over $Z$, 
$$
\vol(K_H+\Gamma_H)=\vol(K_G+\Gamma_G)
$$ 
which means that $\vol(K_H+\Gamma_H)$ is bounded by $v$ as well.

\emph{Step 5.}
Let $Y\to Y'$ be the contraction$/Z$ defined by $K_Y+\Gamma_Y$. Since  $K_Y+\Gamma_Y$ is big$/Z$, 
$Y\to Y'$ is birational. 
Let $T$ be a component of $\Supp \rddown{\Gamma_Y}=\Supp h^*D$ that is not contracted over $Y'$, that is, 
$$
K_T+\Gamma_T:=(K_Y+\Gamma_Y)|_T
$$
is big. Since the coefficients of $\Gamma_Y$ belong to the fixed finite set $\{b,1\}$, 
the coefficients of $\Gamma_T$ belong to a DCC set $\Omega$ which only depend on $b$ and $d$ [\ref{log-flips}, Corollary 3.10]. 
By  [\ref{HMX}, Theorem 1.3], 
there is a real number $\theta>0$ depending only on $d$ and $\Omega$ such that 
$\vol(K_T+\Gamma_T)\ge \theta$. 

By Step 4, we can assume that $\vol(K_H+\Gamma_H)\le v$ for every fibre $H$ of $h$ other than $h^*D$.  
Now, write $h^*D=\sum m_iT_i$ where $T_i$ are the irreducible components 
of $h^*D$.  By Lemma \ref{l-intersection}, we get the following 
equalities of intersection numbers 
\begin{equation*}
\begin{split}
\vol(K_H+\Gamma_H) & =(K_H+\Gamma_H)^{d-1}\\
 & =((K_Y+\Gamma_Y)|_{H})^{d-1}\\
 & =\sum m_i((K_Y+\Gamma_Y)|_{T_i})^{d-1}\\
 & =\sum m_i \vol(K_{T_i}+\Gamma_{T_i})\\
\end{split}
\end{equation*}

 Therefore if $K_{T_i}+\Gamma_{T_i}$ is big for some $i$, then 
$$
\vol(K_H+\Gamma_H)\ge m_i \vol(K_{T_i}+\Gamma_{T_i})\ge m_i\theta
$$
which implies that such $m_i$ are bounded by $\frac{v}{\theta}$. However, we do not get any bound for $m_j$  
if $K_{T_j}+\Gamma_{T_j}$ is not big. We will try to get rid of such $T_j$.

\emph{Step 6.} 
Run the 
LMMP$/Y'$ on $K_Y+\Delta_Y$ with scaling of $P_Y:=\Gamma_Y-\Delta_Y$. Note that $K_Y+\Gamma_Y$ is numerically 
trivial on each extremal ray contracted in the process since $K_Y+\Gamma_Y\equiv 0/Y'$. The LMMP terminates for reasons similar to Step 4.
In some step of the LMMP, we arrive on a model on which the 
birational transform of $K_Y+\Gamma_Y-rP_Y$ is semi-ample$/Y'$ where $r>0$ is a small 
number. Replace $Y$ with that model. Since $K_Y+\Gamma_Y$ is semi-ample$/Z$, 
$K_Y+\Gamma_Y-rP_Y$ is also semi-ample$/Z$ if we take $r$ to be sufficiently small.

Let $Y\to Y''$ be the contraction$/Z$ defined by $K_Y+\Gamma_Y-rP_Y$. Since $r$ is 
 sufficiently small, the map $Y''\bir Y'$ is actually a morphism.   
On the other hand, since,  by Step 2, $P_Y$ and $\rddown{\Gamma_Y}$ have no common components, any component of 
$\rddown{\Gamma_Y}$ that is contracted over $Y'$ is also contracted over $Y''$: indeed if 
$T$ is a component of $\rddown{\Gamma_Y}$ that is contracted over $Y'$, then $(K_Y+\Gamma_Y)|_T$ 
is not big hence $(K_Y+\Gamma_Y-rP_Y)|_T$ is also not big; so $T$ should be contracted over $Y''$ 
as well. Thus, by Step 5, the coefficient of any component of  $h''^*D$ is bounded by $\frac{v}{\theta}$
where $h''$ is the morphism $Y''\to Z$. 

\emph{Step 7.} 
Since  $K_{Y''}+\Gamma_{Y''}$ 
and $K_{Y''}+\Gamma_{Y''}-rP_{Y''}$ are both $\R$-Cartier, $P_{Y''}$ 
is also $\R$-Cartier which in turn implies that $K_{Y''}+\Delta_{Y''}$ is $\R$-Cartier as well. 
In particular, $K_{Y''}+\Delta_{Y''}$ is lc.
Let $\pi\colon N\to X$ and $\mu\colon N\to Y''$ 
be a common resolution. Since  $(X,B)$ is $\epsilon$-lc, we can write 
$$
K_{Y''}+B_{Y''}:=\mu_*\pi^*(K_X+B)
$$
where each coefficient of $B_{Y''}$ is $\le 1-\epsilon$. On the other hand, 
if we write 
$$
K_{W}+B_{W}=\phi^*(K_X+B)
$$
as in Step 2 then $B_W+\epsilon \rddown{\Delta_{W}} \le \Delta_{W}$ which implies that 
$B_{Y''}+\epsilon \rddown{\Delta_{Y''}} \le \Delta_{Y''}$ hence
$$
K_{Y''}+B_{Y''}+ \epsilon \rddown{\Delta_{Y''}}\le K_{Y''}+\Delta_{Y''}
$$ 
Moreover, $\rddown{\Delta_{Y''}}=\Supp h''^*D$.

By Steps 5 and 6, we can write $h''^*D=\sum m_iT_i$ where the $m_i$ are all bounded by $\frac{v}{\theta}$. 
This implies that there is a real number $\delta>0$ depending only on $d,\epsilon,\mathcal{P}$ such that 
$$
K_{Y''}+B_{Y''}+\delta h''^*D\le K_{Y''}+\Delta_{Y''}
$$ 
Therefore, $K_{Y''}+B_{Y''}+\delta h''^*D$ is lc which in turn implies that $K_X+B+\delta f^*D$ is lc because 
$$
\pi^*(K_{X}+B+\delta f^*D)=\mu^*(K_{Y''}+B_{Y''}+\delta h''^*D)
$$ 
where we use the fact that $K_X+B\sim_\R 0/Z$.
Note however that $B_{Y''}+\delta h''^*D$ is not necessarily effective but it is a sub-boundary.
Finally, if $t$ is the lc threshold of $f^*D$ with respect to $(X,B)$, then we have $t\ge \delta$. 
In other words, the coefficient of $D$ in $B_Z$ is at most $1-\delta$ and this completes the proof of the 
proposition. 
\end{proof}

Next we bound the coefficients of the discriminant b-divisor $\mathcal{B}_Z$ (see \ref{s-adjunction} 
for definitions). But first we need a couple of lemmas. 

\begin{lem}\label{l-compact}
Assume that $\epsilon>\epsilon'>0$. Let $(X,B)$ be a $\Q$-factorial $\epsilon$-lc pair and $f\colon X\to Z$ 
a contraction such that $K_X+B\sim_\R 0/Z$ and $B$ is big$/Z$. Then, there are normal  projective varieties 
${X}'\supseteq X$ and ${Z}'\supseteq Z$ and a contraction 
${f}'\colon {X}'\to {Z}'$ such that

$\bullet$ $({X}',{B}')$ is $\Q$-factorial $\epsilon'$-lc and $K_{{X}'}+{{B}'}\sim_\R 0/{Z}'$,

$\bullet$ $(K_{{X}'}+{{B}'})|_X=K_X+B$, and ${f}'|_X=f$.
\end{lem}
\begin{proof}
We can compactify the morphism $f$ to a morphism $\overline{f}\colon \overline{X}\to \overline{Z}$ 
where $\overline{X}$ and  $\overline{Z}$ are normal projective varieties. 
Let $\phi\colon W\to \overline{X}$ be a log resolution and let 
$B_W=\overline{B}^{\sim}+(1-\epsilon')E$ where $E$ is the reduced exceptional divisor of $\phi$, 
$\overline{B}$ is the closure of $B$ in $\overline{X}$, and $\overline{B}^{\sim}$ is the birational 
transform of $\overline{B}$.

Run the LMMP$/\overline{X}$ on $K_W+B_W$ with scaling of some ample divisor. 
We get a model $V$ on which $K_{V}+B_V$ is nef$/\overline{X}$. Since $(X,B)$ is $\epsilon$-lc, over $X$,
 we can write 
$K_W+B_W\equiv G$ where $G$ is effective and exceptional. So, the LMMP contracts every component of $G$  
over $X$ and since $X$ is $\Q$-factorial, $V\to \overline{X}$ is an isomorphism over $X$.
Now, run the LMMP$/\overline{Z}$ on $K_V+B_V$ with scaling of some ample divisor. 
Since $B_V$ is big$/\overline{Z}$, 
the LMMP terminates with a model $X'$ on which $K_{X'}+B'$ is semi-ample$/\overline{Z}$. 

Let 
${f}'\colon X'\to Z'/\overline{Z}$ be the contraction defined by $K_{X'}+B'$. 
Since $V\to \overline{X}$ is an isomorphism over $X$ and $K_X+B\sim_\R 0/Z$, 
the map $V\bir X'$ is an isomorphism over $Z$ hence  the morphism 
$X'\to Z'$ coincides with $X\to Z$ over $Z$. 
By construction, $(X',B')$ and $f'$ satisfy all the properties of the lemma. 
\end{proof}

\begin{lem}\label{l-klt}
Let $(X,B)$ be a klt pair and $f\colon X\to Z$ 
a contraction such that $K_X+B\sim_\R 0/Z$ and $B$ is big$/Z$. 
Assume that $D$ is a component of the discriminant b-divisor $\mathcal{B}_{Z}$ with positive coefficient. 
Then, there is an extremal contraction $Z''\to Z$ which extracts $D$. 
\end{lem}
\begin{proof} 
By assumptions, there is a birational contraction $g\colon Z'\to Z$ such that $D$ is a component of $B_{Z'}$ 
with positive coefficient.
By taking a $\Q$-factorialisation we may assume that $X$ is $\Q$-factorial.
By Lemma \ref{l-compact}, we may assume that $X,Z$ are projective. 
Assume that $B$ has rational coefficients. We will argue as in [\ref{am05}]. 
We may assume that $Z'$ is an Ambro model and that  $(Z',B_{Z'})$ is log smooth. 
In particular, $M_{Z'}$ is the pullback of a nef and big divisor via some 
contraction $Z'\to T$. Since $(Z',B_{Z'})$ is a klt sub-pair we can pick $M_{Z'}$ 
so that $M_{Z'}\ge 0$ and that $(Z',B_{Z'}+M_{Z'})$ is again a klt sub-pair. 
Put  $M_Z=g_*M_{Z'}$. Then, $(Z,B_Z+M_Z)$ is a klt pair 
because 
$$
K_{Z'}+B_{Z'}+M_{Z'}=g^*(K_Z+B_Z+M_Z)
$$ 
Now, $a(D,Z,B_Z+M_Z)<1$ hence there is an extremal contraction $Z''\to Z$ which extracts $D$.

If $B$ is not a rational boundary, we can  
approximate $B$ with rational boundaries (cf. Fujino-Gongyo 
[\ref{Fujino-Gongyo-adjunction}, Theorem 3.1]). More precisely, using Shokurov's polytopes [\ref{log-models}],  
we can find a rational boundary $C$ sufficiently close to $B$ 
such that $(X,C)$ is klt, $K_X+C\sim_\R 0/Z$, $C$ is big$/Z$, 
and the coefficient of $D$ in the discriminant b-divisor $\mathcal{C}_{Z}$ is positive. Now apply the arguments above for the 
case of rational boundaries. 
\end{proof}

\begin{prop}\label{p-1}
Assume that $d,\epsilon, \mathcal{P}$ are as in Conjecture $S_{d,\epsilon, \mathcal{P}}$ and that 
$\mathcal{P}$ is a bounded family of couples. Then, there is a 
real number $\delta>0$ depending only on $d,\epsilon, \mathcal{P}$ such that if $(X,B)$ and $f\colon X\to Z$ 
are as in the conjecture, then every coefficient of $\mathcal{B}_Z$ is $\le 1-\delta$. 
\end{prop}
\begin{proof}
Pick $\epsilon'\in (0,\epsilon)$ and let $\delta\in(0,1)$ be as in Proposition \ref{p-2} for 
the data $d,\epsilon', \mathcal{P}$. 
Let $(X,B)$ and $f\colon X\to Z$ be as in Conjecture $S_{d,\epsilon, \mathcal{P}}$. 
By taking a $\Q$-factorialisation, we may assume that $X$ is $\Q$-factorial. This does not affect  $\mathcal{B}_Z$.
Assume that some component $E$ of $\mathcal{B}_Z$ has coefficient larger than $1-\delta$. 
 We will derive 
a contradiction. 

By our choice of $\delta$, 
every coefficient of ${B}_Z$ is $\le 1-\delta$, so $E$ is exceptional$/Z$. Since the general fibres of 
$f$ are of Fano type, $B$ is big$/Z$.
By  Lemma \ref{l-klt}, there is an extremal contraction $g\colon Z''\to Z$ 
such that $E$ is the only exceptional divisor of $g$. 
Now $B_{Z''}\ge 0$ and the coefficient of $E$ in $B_{Z''}$ is larger than $1-\delta$.
To get a contradiction with Proposition \ref{p-2} we need to construct a suitable 
fibration over $Z''$.
 
Let $\phi\colon W\to X$ be a log resolution so that the induced rational map $W\to Z''$ 
is a morphism. Let 
$$
\Delta_W:=B^\sim+(1-\epsilon')G 
$$
where $B^\sim$ is the birational transform of $B$ and $G$ is the reduced exceptional divisor of $\phi$. 
The pair $(W,\Delta_W)$ is $\epsilon'$-lc and we can write 
$$
K_W+\Delta_W=\phi^*(K_X+B)+C
$$ 
where $C\ge 0$ and $\Supp C=\Supp G$. In particular, if we run an LMMP on $K_W+\Delta_W$ 
over $X$ (or some open subset of $X$), then the LMMP terminates with $X$ (respectively with 
that open subset).

Let $T$ be the graph of the rational map $X\bir Z''$, that is, $T$ is the closure in $X\times Z''$ of the 
graph of $X_0\to Z''$ where $X_0\subseteq X$ is the domain of  $X\bir Z''$. Since $W$ maps to both $X$ and $Z''$ we 
get an induced morphism $W\to T$. 
Let
$U\subset Z$ be a non-empty open set over which $Z''\to Z$ is an isomorphism. 
Run an LMMP$/T$ on 
$K_W+\Delta_W$ with scaling of some ample divisor. We end up with a model $Y$ on which 
$K_{Y}+\Delta_{Y}$ is nef$/T$. Since $f^{-1}U\subseteq X_0$, the morphism $T\to X$ is an isomorphism 
over $U$. So, 
the morphism $Y\to T$ is also an isomorphism over $U$ by the last paragraph.
Since $K_X+B\sim_\R 0/Z$, $K_{Y}+\Delta_{Y}\sim_\R 0$ over $U$.  

Now, run an LMMP$/Z''$ on 
$K_Y+\Delta_Y$ with scaling of some ample divisor which ends up with a model $X'$ on which 
$K_{X'}+\Delta_{X'}$ is semi-ample$/Z''$ because $\Delta_Y$ is big$/Z''$. Since $K_{Y}+\Delta_{Y}$ is nef over $U$, 
the LMMP does not modify $Y$ over $U$, that is, $Y\bir X'$ is an isomorphism over 
$U$. 

Let $f'\colon X'\to Z'$ be the contraction$/Z''$ defined by 
$K_{X'}+\Delta_{X'}$. Then, the map $Z'\to Z''$ is birational. Moreover, 
$(X',\Delta_{X'})$ and $f'\colon X'\to Z'$ satisfy the assumptions of  
Proposition \ref{p-2} for the data $d,\epsilon', \mathcal{P}$, that is, $(X',\Delta_{X'})$ is $\epsilon'$-lc of dimension $d$, 
$K_{X'}+\Delta_{X'}\sim_\R 0/Z'$, and the log 
general fibres of $(X',\Delta_{X'})$ and $f'$ are the same as the log general fibres of $(X,B)$ and $f$.

Let $\Delta_{Z'}$ be the discriminant on $Z'$ associated to $K_{X'}+\Delta_{X'}$ 
and the fibration $f'$. By Proposition \ref{p-2}, the coefficients of $\Delta_{Z'}$ 
are at most $1-\delta$. 
On the other 
hand, let $B_{Z'}$ be the discriminant on $Z'$ associated to $K_X+B$ and the fibration $f$.  
It is enough to show that the coefficient of $E$ in $B_{Z'}$ is not bigger than the 
 coefficient of $E$ in $\Delta_{Z'}$. 

Let $\pi\colon V\to X$ and $\mu \colon V\to X'$
be a common resolution and let 
$$
M=K_{X'}+\Delta_{X'}-\mu_*\pi^*(K_X+B)
$$
As mentioned above we have 
$$
K_W+\Delta_W - \phi^*(K_X+B)=C\ge 0
$$
Now $M$ is just the pushdown of $K_W+\Delta_W - \phi^*(K_X+B)$ via the rational map $W\bir X'$. 
Therefore, $M\ge 0$. From  $K_X+B\sim_\R 0/Z$ we get 
$$
\pi^*(K_X+B)=\mu^*\mu_*\pi^*(K_X+B)
$$
and this combined with $M\ge 0$ results in 
$$
\mu^*(K_{X'}+\Delta_{X'})-\pi^*(K_X+B)
$$
$$
=\mu^*(K_{X'}+\Delta_{X'})-\mu^*\mu_*\pi^*(K_X+B)
= \mu^*M \ge 0
$$

Since the coefficient of $E$ in $\Delta_{Z'}$ is at most $1-\delta$, over the generic point of $E$ the log divisor   
$K_{X'}+\Delta_{X'}+\delta f'^*E$ is lc. This implies that  
$$
\mu^*(K_{X'}+\Delta_{X'}+\delta f'^*E)=\mu^*(K_{X'}+\Delta_{X'})+\mu^*\delta f'^*E
$$ 
is lc over the generic point of $E$ which in turn implies that 
$$
\pi^*(K_X+B)+\delta\mu^*f'^*E
$$
 is also lc over the generic point of $E$.
Therefore, the coefficient of $E$ in $B_{Z'}$ is at most $1-\delta$. This gives a 
contradiction. Thus, every coefficient of $\mathcal{B}_Z$ is at most $1-\delta$.
\end{proof}

\vspace{0.5cm}
\section{Proof of the main results}

\begin{proof}(of Theorem \ref{t-main}) 
Let $(X,B)$ and $f\colon X\to Z$ be as in Conjecture $S_{d,\epsilon,\mathcal{P}}$. 
If $f$ is birational, then $B_Z=f_*B$ and if we 
take $M_Z=0$, then $K_X+B=f^*(K_Z+B_Z)$ and $(Z,B_Z)$ is $\epsilon$-lc. So, 
$\delta=\epsilon$ works in this case. 
We can then assume that $f$ is not birational.
By taking a $\Q$-factorialisation we may assume that $X$ is $\Q$-factorial. 
The log general fibres may change but only by an isomorphism in codimension one. 
Pick $\epsilon'\in (0,\epsilon)$.
By Lemma \ref{l-compact}, we can replace $X,Z$ with projective varieties  
so that $(X,B)$ and $f\colon X\to Z$ satisfy the assumptions of Conjecture $S_{d,\epsilon',\mathcal{P}}$ 
(note that $\epsilon$ is replaced by $\epsilon'$).

First assume that the coefficients of $B$ are rational numbers. 
Let $\delta'>0$ be the number given by Proposition \ref{p-1} for the data 
$d,\epsilon', \mathcal{P}$, and pick  some $\delta\in (0,\delta')$.
Let $g\colon {Z}'\to Z$ be an Ambro model of $K_X+B$ and $f\colon X\to Z$ 
as defined in \ref{s-adjunction}. 
Since $Z'$ is an Ambro model, $M_{{Z}'}$ 
is nef and good, that is, it is the pullback of a nef and big $\Q$-divisor $M_T$ 
via some contraction $\pi\colon Z'\to T$. 
We can write $M_T\sim_\Q A_{T}+L_{T}$ where $A_{T}$ is  
an ample $\Q$-divisor and $L_{T}\ge 0$. Put $A_{Z'}=\pi^*A_T$ and $L_Z=\pi^*L_T$.
Thus, 
$M_{{Z}'}\sim_\Q A_{{Z}'}+L_{{Z}'}$ where $A_{Z'}$ is semi-ample and $L_{{Z}'}\ge 0$. 
Perhaps after replacing ${{Z}'}$ we may assume that 
$\Supp (B_{{Z}'}+L_{{Z}'})$ has simple normal crossing singularities. By our choice of 
$\delta'$, each coefficient of $B_{Z'}$ is $\le 1-\delta'$.

Since $A_{T}+L_{T}$ is nef and $A_{T}$ is ample, 
$$
A_{T}+\frac{1}{a+1}L_{T}=\frac{a}{a+1}A_{T}+\frac{1}{a+1}(A_{T}+L_{T})
$$ is ample 
for any $a>0$. So, perhaps after replacing $L_{{Z}'}$ with $ \frac{a}{a+1}L_{{Z}'}$ and $A_{{Z}'}$ 
with $A_{{Z}'}+\frac{1}{a+1}L_{Z'}$ 
for some sufficiently small $a>0$, we may assume that the coefficients of $L_{{Z}'}$ are sufficiently small. 
Thus, we can assume that 
the coefficients of $B_{{Z}'}+L_{{Z}'}$ are all $\le 1-\delta$. 
Perhaps after replacing $A_{{Z}'}$ we can also assume that 
$\Supp (B_{{Z}'}+A_{{Z}'}+L_{{Z}'})$ has simple normal crossing singularities and that the coefficients 
of $B_{{Z}'}+A_{{Z}'}+L_{{Z}'}$ are $\le 1-\delta$ and that $A_{{Z}'}\ge 0$. 

Let $A_Z,L_Z$ be the pushdown of $A_{{Z}'},L_{{Z}'}$ respectively.
Since 
$$
K_{{Z}'}+B_{{Z}'}+A_{{Z}'}+L_{{Z}'}=g^*(K_Z+B_Z+A_Z+L_Z)
$$
we deduce that $(Z,B_Z+A_Z+L_Z)$ is a $\delta$-lc pair. By putting $M_Z:=A_Z+L_Z$ 
we finish the proof of the theorem when $B$ has rational coefficients. 

Now we come to the general case, that is, when the coefficients of $B$ are not necessarily 
rational. We will do an approximation to reduce to the rational case. 
Pick $\epsilon''\in (0,\epsilon')$, let $\delta''>0$ be the number given by Proposition \ref{p-1} for the data 
$d,\epsilon'', \mathcal{P}$, and pick  some $\delta\in (0,\delta'')$.
 Then, using Shokurov's polytopes [\ref{log-models}],  
we can find real numbers 
$r_i\ge 0$ and rational boundaries $B^i$ such that 

$\bullet$ $\sum r_i=1$ and $K_X+B=\sum r_i(K_X+B^i)$,

$\bullet$ $K_X+B^i\sim_\Q 0/Z$,

$\bullet$ $\Supp B^i=\Supp B$, 

$\bullet$ $(X,B^i)$ are $\epsilon''$-lc.\\\\ 
Applying the arguments above in the rational case, for each $i$ we can choose 
$M_{Z}^i\ge 0$ such that 
$$ 
K_X+B^i\sim_\Q f^*(K_Z+B_Z^i+M_Z^i) 
$$ 
and that $(Z, B_Z^i+M_Z^i)$ is a $\delta$-lc pair. 

Fix a prime divisor $D$ on $Z$ and let its coefficient in 
$B_Z$ and $B_Z^i$ be $b$ and $b^i$ respectively. By definition, the pair $(X,B^i+(1-b^i)f^*D)$ 
is lc over the generic point of $D$. But then the pair $(X,B+\sum r_i(1-b^i)f^*D)$
is lc over the generic point of $D$. This means that 
$$
1-\sum r_ib^i=\sum r_i(1-b^i)\le 1-b
$$
which in turn implies that $\sum r_ib^i\ge b$. In other words, $\sum r_iB_Z^i\ge B_Z$. 

Now, the pair 
$$
(Z,\sum r_i B_Z^i+ \sum r_i M_Z^i)
$$ 
is $\delta$-lc and by putting 
$$
M_Z:=\sum r_i B_Z^i+ \sum r_i M_Z^i-B_Z
$$ 
we conclude that $(Z,B_Z+M_Z)$ is $\delta$-lc and 
$$
K_X+B\sim_\R f^*(K_Z+B_Z+M_Z)
$$
This completes the proof of the theorem.
\end{proof}

We need the following theorem for the proof of Corollary \ref{c-2}.

\begin{thm}\label{t-bnd-surfaces}
Let  $\epsilon, \lambda>0$ be real numbers.
Consider the set of pairs $(F,B_F)$ introduced just before Corollary \ref{c-2} with the extra assumption 
$d\le 3$. Then, $\mathcal{R}$, the set of the couples $(F,\Supp B_F)$ is a bounded family.
\end{thm}
\begin{proof}
First assume that $\dim F=1$ which means that $F\simeq \PP^1$. 
Replacing $\lambda$ with a smaller number we may assume that it is rational. 
Since $\deg B_F=2$ and each coefficient of $B_F$ is $\ge \lambda$, 
the number of components of $\Supp B_F$ is bounded only depending on $\lambda$. 
Then, the set of the couples $(F,\Supp B_F)$ belongs to a bounded family using Hilbert schemes of 
zero-dimensional subschemes of $F$. 

Now assume that $\dim F=2$.  Let $\Delta_F$ be the boundary obtained from $B_F$ by replacing each 
coefficient with $\lambda$. By definition, $C_F:=B_F-\Delta_F\ge 0$. Pick a small number $t>0$ 
so that $(F,B_F+tC_F)$ is still klt. Run the LMMP on $K_F+B_F+tC_F\sim_\R tC_F$. 
We get a log minimal model on which the pushdown of $C_F$ is nef. By replacing $F$ 
with that model we could assume that $C_F$ is nef. Note that the LMMP can contract only 
the components of $B_F$ so we can pullback boundedness to the original setting. 

By Alexeev [\ref{Alexeev}] the varieties $F$ belong to a bounded family. 
In particular, the Cartier index of $K_F+\Delta_F$ is bounded. So, we can 
pick a $\Q$-divisor $A_F\ge 0$ such that $K_F+\Delta_F+A_F\sim_\Q 0$ and 
such that the coefficients of  $\Delta_F+A_F$ belong to a fixed finite set 
depending only on $\epsilon$ and $\lambda$. Now, apply  [\ref{HMX}, Corollary 1.7] to get the boundedness 
of the couples $(F,\Supp \Delta_F)=(F,\Supp B_F)$.
\end{proof}

\begin{proof}(of Corollary \ref{c-2}) 
This follows from Theorems \ref{t-bnd-surfaces} and \ref{t-main}. 
\end{proof}

\begin{rem}\label{rem-bnd-R}
Actually, if one tries to prove the boundedness of $\mathcal{R}$ for $d>3$ inductively, then variants of 
Theorem \ref{t-main} 
appear naturally in the induction process. More precisely: pick $(F,B_F)$ as defined  just before Corollary \ref{c-2}.
 After running a suitable LMMP 
one can assume that $F$ has a Mori fibre space structure $F\to G$. Assume that $\dim G>0$. The couples associated to the 
log general fibres of this morphism 
are bounded by induction. One uses a variant of Theorem \ref{t-main} to show that $G$ is also bounded 
(one needs a stronger form of \ref{t-main} in which the coefficients of 
$B_G+M_G$ are not too small and this needs an effective version of Ambro's result on $M_G$); 
the next step is to use these boundedness results to prove that 
$F$ itself together with $\Supp B_F$ are bounded. This is more related to the work of Shokurov mentioned 
in the introduction. On the other hand, if $\dim G=0$, one needs different arguments.
\end{rem}

\begin{proof}(of Corollary \ref{c-main}) 
This is immediate by Theorem \ref{t-main} and [\ref{HMX}, Corollary 1.7].
\end{proof}

\begin{proof}(of Corollary \ref{c-dim2}) 
Let $(X,B)$ and $f\colon X\to Z$ be as in Conjecture $S_{d,\epsilon,\mathcal{P}}$ such that 
$d\le \dim Z+1$. We may assume that $d=\dim Z+1$ otherwise 
$f$ is birational and we can argue as in the proof of Theorem \ref{t-main}. Using the arguments of the proof 
of Lemma \ref{l-dim-1}, Proposition \ref{p-1}, and Theorem \ref{t-main}, 
we can reduce the problem to the case $\dim Z=1$ and $\dim X=2$. Moreover, we only need to 
show that the coefficients of $B_Z$ are $\le 1-\delta$ for a fixed $\delta>0$ depending only on 
$\epsilon$.

By replacing $X$ with its minimal resolution, we may assume that $X$ is smooth. Next by running an 
LMMP$/Z$ on $K_X$ we can assume that $X\to Z$ is an extremal contraction, that is, in this case 
a $\PP^1$-bundle. Fix $\epsilon'\in (0,\epsilon)$.
Pick a closed point $D\in Z$ and let $t$ be the $\epsilon'$-lc threshold of $f^*D$ with respect to the 
pair $(X,B)$, that is, $t$ is the largest number so that $(X,B+tf^*D)$ is $\epsilon'$-lc. 
Here $T:=f^*D$ is a reduced curve. 
If the coefficient of $T$ in ${B+tf^*D}$ is $1-\epsilon'$,  then $t\ge (\epsilon-\epsilon')$ because the 
coefficient of $T$ in $B$ is at most $1-\epsilon$. In this case, the coefficient of $D$ 
in $B_Z$ is at most $1-\epsilon+\epsilon'$ and $\delta=\epsilon-\epsilon'$ works. From now on we
assume that  the coefficient of $T$ in ${B+tf^*D}$ is $<1-\epsilon'$. 

There is a prime exceptional$/X$ divisor $E$ such that $a(E,X,B+tf^*D)=\epsilon'$. Let $Y\to X$ be the 
extremal contraction which extracts $E$. There is another extremal ray on $Y/Z$ which we can 
contract to get $X'$, and $X'\to Z$ is an extremal contraction. However, $X'$ may not be smooth. 
Let $g,h,f'$ denote the contractions $Y\to X$, $Y\to X'$, and $X'\to Z$ respectively. 
Write $K_{X'}+B'=h_*g^*(K_X+B)$ and let $T'=\Supp f'^*D$. Then, the coefficient of $T'$ in 
$B'+tf'^*D$ is $1-\epsilon'$. Since the coefficient of $T'$ in $B'$ is at most $1-\epsilon$, 
it is enough to show that $f'^*D=m'T'$ where $m'$ is bounded depending on $\epsilon$. 

Applying the boundedness of $\epsilon'$-lc complements 
in dimension two [\ref{B-complements}], we get a real number $\epsilon''>0$ and a finite set 
$\Lambda$ of rational numbers depending only on $\epsilon'$ which satisfy: 
there is a boundary $\Delta'$ such that $(X',\Delta')$ is $\epsilon''$-lc, $K_{X'}+\Delta'\sim_\Q 0/Z$, 
and the coefficients of $\Delta'$ belong to $\Lambda$. 
Now we can apply Corollary \ref{c-2} to get the bound on $m'$. 
\end{proof}

\begin{rem}\label{rem-not-bnd}
An obvious question to ask is: what can we do about Conjecture $S_{d,\epsilon,\mathcal{P}}$ 
if $\mathcal{P}$ is not bounded?  
Let $(X,B)$ and $f\colon X\to Z$ be as in Conjecture $S_{d,\epsilon,\mathcal{P}}$, and let $(F,B_F)$ be
a log general fibre. 
In some cases, the couples $(F,\Supp B_F)$ belong to a bounded family $\mathcal{P}'$ even if $\mathcal{P}$ 
is not bounded. For example, all the couples in $\mathcal{P}$ of dimension $\le d-1$ may belong to 
a bounded family $\mathcal{P}'$, that is, the unbounded part of $\mathcal{P}$ 
may not be relevant to the conjecture. In this case, Conjecture $S_{d,\epsilon,\mathcal{P}'}$ implies 
Conjecture $S_{d,\epsilon,\mathcal{P}}$, and we can use Theorem \ref{t-main} to prove Conjecture 
$S_{d,\epsilon,\mathcal{P}'}$. 

In general, we cannot shrink $\mathcal{P}$ to a bounded family $\mathcal{P}'$. The idea then is to 
modify the pair $(X,B)$ to get boundedness. For example, as in the proof of Corollary \ref{c-dim2}, 
one can hope to find a real number $\epsilon'>0$ and a fixed finite set $\Lambda$ of real numbers depending 
only on $d,\epsilon'$ such that: there is a boundary $\Delta$ so that $(X,\Delta)$ is $\epsilon'$-lc, $K_{X}+\Delta\sim_\Q 0/Z$, 
and the coefficients of $\Delta$ belong to $\Lambda$. Next one applies Corollary \ref{c-main}. This is closely related 
to the theory of complements [\ref{Sh-complements}][\ref{B-complements}].
\end{rem}


\vspace{2cm}

\flushleft{DPMMS}, Centre for Mathematical Sciences,\\
Cambridge University,\\
Wilberforce Road,\\
Cambridge, CB3 0WB,\\
UK\\
email: c.birkar@dpmms.cam.ac.uk

\end{document}